\documentclass[12pt]{article}
\usepackage{times}
\usepackage{mathptmx, amsmath, epsfig,graphicx,rotating}
\usepackage{dsfont,stmaryrd, bbm, wasysym, pxfonts, txfonts, amssymb, epstopdf}
\usepackage{lineno}

\newtheorem{theorem}{Theorem}

\newtheorem{definition}[theorem]{Definition}

\newenvironment{proof}{\trivlist\item[]\textbf{Proof.\ }}
                      {\endtrivlist}

\begin{document}

\title{Optimal Scale Invariant Wigner Spectrum Estimation of Gaussian Locally Self-Similar Processes Using Hermite Functions}

\author{Yasaman Maleki}

\maketitle

\begin{abstract}
This paper investigates the mean square error optimal estimation of scale invariant Wigner spectrum for the class of Gaussian locally self-similar processes, by the multitaper method. In this method, the spectrum is estimated as a weighted sum of scale invariant windowed spectrograms. Moreover, it is shown that the optimal multitapers are approximated by the quasi Lamperti transformation of Hermite functions, which is computationally more efficient. Finally, the performance and accuracy of the estimation is studied via simulation.

\noindent{\bf{keywords:}}Locally self-similar processes \and scale invariant Wigner spectrum \and multitaper method \and Hermite functions \and time-frequency analysis.

\noindent{\bf{60G18 \and 60G99}}
\end{abstract}

\section{Introduction}
\label{intro}
Spectrum estimation of non-stationary processes is one of the most important problems in time-frequency analysis. Some definitions have been proposed for non-stationary spectrum, where the most notable one is the Wigner-Ville spectrum (WVS) proposed by Martin \cite{martin} for the class of harmonizable processes. Several methods have been proposed for WVS estimation, such as: the multitaper method \cite{bayram} and the mean square error (MSE) optimal kernel method \cite{say}, which is proposed  for the class of Gaussian harmonizable processes. Moreover, in \cite{wahlberg2}, it is shown that the WVS can be estimated as a weighted sum of spectrograms; also, the WVS can be approximated by a set of Hermite functions for the class of locally stationary processes. Such an approximation is advantageous when it comes to calculation, since only a limited number of Hermite functions need to be calculated \cite{sandsten}.

However, these methods are insufficient for scale invariant (or self-similar) processes, a subclass of non-stationary processes, that occur in important applications such as: turbulence, hydrology, telecommunications network traffic and image processing. Besides, sometimes it happens that self-similar processes are not quite adequate for real world phenomena, and it would be useful to consider more general classes of such processes. Locally self-similar processes (LSSPs) constitute an extensions of self-similar processes. This class of processes can be used to describe physical systems for which statistical characteristics change slowly in time. Flandrin \cite{fla1}, introduced locally self-similar processes as Lamperti transformation of locally stationary processes in Silverman's sense \cite{silverman}. In such situations, some specific tools are needed in time-frequency analysis, that are different from other non-stationary ones and compatible with the scale invariant property. Therefore, a particular type of the Wigner spectrum should be considered.
The scale invariant Wigner spectrum (SIWS) of a process $\{X(t), t \in \mathbbm{R}^+\}$ is defined as:
$$W_{E,X}(t, \xi)  = E \big{\{}\int_0^{\infty} X(t\sqrt{\tau})X^*(t/\sqrt{\tau}) \tau^{-i2\pi \xi-1} d\tau \big{\}}$$
\begin{equation}\label{exactsiws}
\;\qquad = \int_0^{\infty} R_X(t\sqrt{\tau}, t/\sqrt{\tau}) \tau^{-i2\pi \xi-1} d\tau,
\end{equation}
where $E$ denotes the expectation operator and $R_X(t,s) = E(X(t)X^*(s))$ \cite{fla5}, \cite{fla}. The integral inside the expectation operator is a stochastic integral which is called the scale invariant Wigner distribution (SIWD), and will be interpreted as a mean-square integral. The interchange of expectation and integration in the second equation is justified if the above-mentioned stochastic integral exists in the mean-square sense \cite{say}. A necessary and sufficient condition for its existence is
$$\int_0^{\infty} \int_0^{\infty} E\{A_X(t, \tau_1) A^*_X(t, \tau_2)\} (\tau_1/\tau_2)^{-i2\pi \xi} \frac{d\tau_1}{\tau_1} \frac{d\tau_2}{\tau_2} < \infty,$$
for all $(t, \xi)$, where $A_X(t, \tau) := X(t\sqrt{\tau})X^*(t/\sqrt{\tau})$ \cite{say}.

The scale invariant Wigner spectrum describes the time evolution of Mellin variable in a similar way as the Wigner-Ville spectrum does for the (Fourier) frequency variable \cite{fla}.

The scale invariant Wigner spectrum is estimated  in \cite{maleki} for the class of Gaussian locally self-similar processes, using the mean square optimal kernel method. The estimation is based on finding the optimal kernel in the Cohen's class counterpart \cite{fla5} of time-frequency representations. However, the scale invariant feature of the process and consequently, the spectrum, makes the SIWS estimation method different from the other non-stationary ones. So, the non-stationary spectrum estimators  should be modified to be reconciled to SIWS estimation problem.

The estimation of the scale invariant Wigner spectrum, using the optimal kernel method can be simplified by kernel decomposition and calculating multitaper spectrograms. In this paper, the SIWS is estimated  as a weighted sum of spectrograms of the data with different eigenvectors as sliding windows and the eigenvalues as weights. Also, the discrete-time multitapers corresponding to the mean square error optimal kernel for a class of locally self-similar processes are computed and the performance of the resulting estimator is compared to the scale invariant Wigner distribution, which is a classical estimator of the SIWS. Furthermore, it is shown that the optimal multitapers are well approximated by the quasi Lamperti transformation of Hermite functions, that limited number of windows can be used for a mean square error optimal spectrogram estimate. Moreover, the superiority of using the quasi Lamperti transform of Hermite functions in scale invariant Wigner spectrum estimation over using the scale invariant Wigner distribution is studied via simulations.

The results of this paper are true when the covariance function of a Gaussian process and the kernel $\Phi$ have certain regularity properties, meaning a certain amount of smoothness and asymptotic decay \cite{sandsten}. A sufficient condition is, $R_X \in {S}_0(\mathbbm{R}^{+}\times \mathbbm{R}^{+})$ and $\Phi \in {S}_0(\mathbbm{R}^{+} \times \mathbbm{R})$, which denotes Feichtinger’s algebra \cite{banach}, \cite{segal}. In this case, we define $S_0$ on the locally compact abelian (LCA) groups $\mathbbm{R}^{+}\times \mathbbm{R}^{+}$ and $\mathbbm{R}^{+}\times \mathbbm{R}$. So, by the fact that the Mellin transform plays the same central role for self-similarity as the Fourier transform plays for stationarity and also, general nonstationary methods built on Fourier representation lead to corresponding methods for Mellin representation \cite{fla1}, \cite{fla}, we conclude that the abstract Fourier transform becomes the Mellin transform in the first case, and a mixed Mellin-Fourier transform in the second case. So, the Feichtinger’s results [12] justify the invariance properties provided in the paper.
Under this condition the SIWD and ambiguity process and also, the Cohen’s class counterpart of time-frequency representations are all second-order processes \cite{wahl3}. Furthermore, the estimation performance criterion applied in \cite{maleki}, consisting of the integral of the mean square error of Cohen's class counterpart, is finite for $(t, \xi) \in \mathbbm{R}^+ \times \mathbbm{R}$ \cite{wahl2}. Also, the scale invariant Wigner distribution is well defined more generally for every multiplicative harmonizable Gaussian process \cite{fla1}, \cite{fla}, \cite{wahl4}.

The paper is organized as follows. Next section, presents the background on locally self-similar, locally self-similar chirp, multicomponent locally self-similar and circularly symmetric processes. Section \ref{optimal-kernel}, gives a brief review of the optimal kernel method in scale invariant Wigner spectrum estimation. Section \ref{multitaper}, discuss the multiple window spectrum estimation of scale invariant Wigner spectrum. Section \ref{approximation}, concerns the SIWS estimation using the quasi Lamperti transform of Hermite functions. Finally, Section \ref{evaluation}, evaluates the proposed approximation methods. Some examples are also presented throughout the paper which illustrate the premier performance of the techniques.

\section{Preliminaries}\label{preliminaries}
\begin{definition}
A locally self-similar process (LSSP) \cite{fla1}, is a complex-valued stochastic process whose covariance function has the form
\begin{equation}
R_X(t,s)=(ts)^H q(\ln \sqrt{ts}) C(\frac{t}{s})\label{1''}
\end{equation}
$$\qquad\quad= Q \varotimes C \circ \kappa^{-1}(t,s),$$
where $q$ and $C$ are complex-valued functions, $q$ must have a constant sign which we assume positive, and $C$ is a non-negative definite function; $Q(t):=t^{2H}q(\ln t)$, and $\kappa$ denotes the coordinate transformation as
$$\kappa(t,s):=(t\sqrt{s},t/\sqrt{s})\Longleftrightarrow \kappa^{-1}(t,s)=(\sqrt{ts}, \frac{t}{s}).$$

We also assume that $Q$ and $C$ are continuous. Continuity of $Q$ and $C$ is equivalent to mean-square continuity of the process $X(t)$ \cite{maleki}. In fact, everywhere continuity of $R_X$ is implied by continuity on the diagonal \cite{loeve}. Thus, if $Q$ is continuous everywhere and $C$ is continuous in $1$, then $R_X$ is continuous everywhere and the process is mean square continuous \cite{maleki}.

A locally self-similar chirp process (LSSCP) \cite{maleki}, is a Gaussian circularly symmetric process with covariance function
\begin{equation}\label{lsscp}
R_X(t, s) = Q(\sqrt{ts})C(t/s) (t/s)^{ia(\ln \sqrt{ts}-b)}, \qquad a, b \in \mathbbm{R}
\end{equation}
where the constant $a$ determines the chirp frequency, and $b$ the start of the chirp frequency.

The definition of LSSP is also extended to a sum of locally self-similar processes.
A multicomponent locally self-similar process (MLSSP) \cite{maleki}, is a process whose covariance function has the form
\begin{equation}
R_X(t, s) = \sum_{j=1}^{\infty} Q_j(\sqrt{ts}) C_j (t/s),\label{mlssp}
\end{equation}
where each term $Q_j \varotimes C_j \circ \kappa^{-1}$ is the covariance function of a LSSP.
\end{definition}
\vspace{0.1 in}
\input{epsf}
\epsfxsize=3.5in \epsfysize=1.5in
\begin{figure}\label{thekernels1}
\vspace{0.2 in}
\includegraphics[width=1\textwidth]{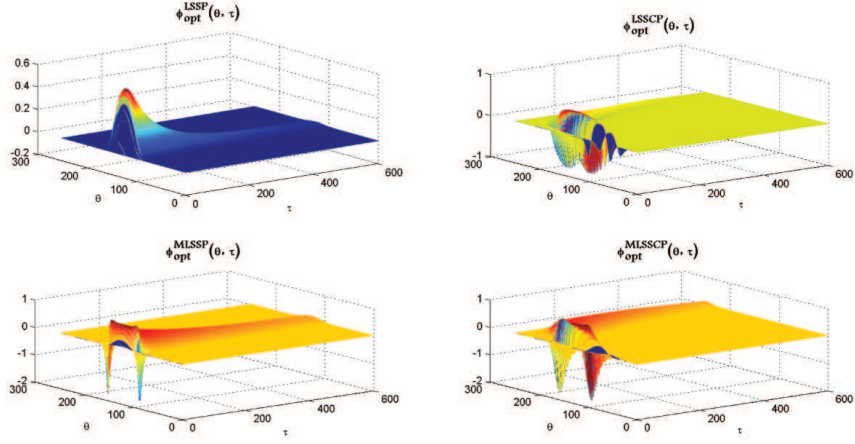}
\vspace{-0.1in}
\caption{\scriptsize The mean square error optimal ambiguity domain kernel of LSSP, LSSCP, MLSSP and MLSSCP with covariance functions (\ref{1''})-(\ref{mlssp}). Top left: LSSP, $c = 2$. Top right: LSSCP, $c = 2, a = 2, b = 0$. Bottom left: MLLSP, $c_1 = 4, c_2 = 10, H_1= 0.2, H_2 = 0.8$. Bottom right:
MLSSCP, $c_1 = 4, c_2 = 10, a = 2, b = 0, H_1= 0.2, H_2 = 0.8$.}
\end{figure}
\vspace{-0.1in}

\begin{definition}
For a circularly symmetric or proper process \cite{pic}, \cite{sch} $X$, the processes
$$\{e^{i\theta} X(t)\}_{\theta \in [0,2\pi)},$$
are identically distributed for all $\theta \in [0,2\pi)$.
\end{definition}

\begin{definition}
Let $g(t)$ be a piecewise continuous and rapidly decaying at both $0$ and $\infty$, i.e., the function $t^A g(t)$ is bounded on $\mathbbm{R}^+$ for any $A \in \mathbbm{R}$. Then, the integral
$$\widetilde{g}(s) := (\mathcal{M}g)(s) = \int_0^{\infty} g(t) t^{-s-1} dt,$$
converges for any complex value of $s$ and defines a holomorphic function of $s$, called the Mellin transform of $g(t)$. If $\int_0^{\infty} |g(t)| t^{-\Re s-1}dt < \infty$, the transform $\widetilde{g}(s)$ exists  \cite{zagier}, where $\Re s = Re(s)$ and the integrability condition is satisfied automatically due to the assumptions.

For a function $g$ of several variables, we denote partial Mellin transform with respect to
variables indexed by $j, k$, with $\mathcal{M}_{j,k}g$.
\end{definition}

\section{Optimal Kernel SIWS Estimation}\label{optimal-kernel}
The scale invariant Wigner spectrum is estimated in \cite{maleki}, for zero-mean, real-valued and complex-valued Gaussian circularly symmetric  locally self-similar processes by deriving the minimum mean square error optimal kernel within Cohen's class counterpart of time-frequency representations (TFRs). The Cohen's class counterpart, closely parallels the conventional Cohen's class \cite{Cohen} and shares with it some of its most interesting properties, namely those concerning the usefulness and versatility of distributions associated to separable smoothing functions \cite{fla5}.

The stochastic Riemann integral
\begin{equation}
P_X(t, \xi)=\int_{-\infty}^\infty \int_0^\infty W_X(\frac{t}{s},\xi-\eta)\Phi(s, \eta) \frac{ds}{s} d\eta,\label{fla}
\end{equation}
\noindent where $W_X(t,\xi) = \int_{0}^{\infty} X(t \sqrt{\tau}) X^*(t /\sqrt{\tau}) \tau^{-i2\pi \xi-1} d\tau $ is the SIWD and $\Phi$ is the 2-D kernel that completely characterizes the particular TFR $P_X$ \cite{fla5}, by definition, is a member of Cohen's class counterpart of time-frequency representations \cite{fla5}. Eq (\ref{fla}) can equivalently be represented as
\begin{equation}\label{cohen}
P_X(t,\xi)=\int_{-\infty}^\infty \int_0^\infty A_X(\theta, \tau)\phi(\theta, \tau)t^{i2\pi\theta}\tau^{-i2\pi\xi-1} d\tau d\theta
\end{equation}
$$ = \mathcal{M}_1^{-1}\mathcal{M}_2\{A_X(\theta, \tau)\phi(\theta, \tau)\},\quad\quad\;\;\;$$
where $\phi(\theta, \tau)=\int_{-\infty}^\infty \int_0^\infty \Phi(t,\xi) t^{-i2\pi\theta-1} \tau^{i2\pi\xi} dt d\xi$, and $A_X(\theta, \tau)$ is the scale invariant ambiguity process \cite{fla5}, which can be obtained by Mellin duality of the scale invariant Wigner distribution, as:
$$A_X(\theta, \tau) = \int_0^{\infty} X(t\sqrt{\tau})X^*(t/\sqrt{\tau}) t^{-i2\pi \theta-1} dt$$
$$ = \mathcal{M}_1\mathcal{M}_2^{-1} W_X(t,\xi).\qquad\;$$

For a zero mean real-valued, or complex-valued circularly symmetric Gaussian process $\{X(t), t>0\}$ where the covariance function  $R_X$ belongs to the Feichtinger algebra $S_0(\mathbbm{R}^{+}\times \mathbbm{R}^{+})$, the stochastic Riemann integrals $W_X$ and $A_X$ exist for all argument values and will be interpreted as mean-square (m.s.) integrals \cite{say}, \cite{loeve}. Thus, by \cite{wahl3}, \cite{wahl2} and the correspondence between the Fourier and the Mellin transforms \cite{fla1}, \cite{fla}, $W_X$ and  $A_X$ are second order stochastic processes. Also, $P_X$ is a second-order stochastic process for all $(t,\xi) \in \mathbbm{R}^+\times \mathbbm{R}$. Furthermore, the kernel $\phi$ is a forward-backward Mellin transform of $\Phi \in S_0(\mathbbm{R}^{+}\times \mathbbm{R})$. Thus, when we work on the LCA group $\mathbbm{R}^{+}\times \mathbbm{R}$, the abstract Fourier transform in \cite{segal} reduces to the Fourier-Mellin transform and $\phi \in S_0(\mathbbm{R}\times \mathbbm{R}^{+})$.

Since the Cohen's class counterpart is completely characterized in terms of kernels, so the SIWS estimation is based on finding the optimal kernel which minimizes the mean square error integral
$$J(\phi)= \int_{-\infty}^{\infty} \int_0^{\infty}  E \big {|}P_X(t,\xi)-W_{E,X}(t,\xi)\big{|}^2 dt d\xi.$$
If $R_X \in S_0(\mathbbm{R}^{+}\times \mathbbm{R}^{+})$ and $\Phi \in S_0(\mathbbm{R}^+\times \mathbbm{R})$, then by the correspondence between the Fourier and Mellin transforms and the results in \cite{wahl3}, we have that $W_{E,X}(t,\xi) \in S_0(\mathbbm{R}^+\times \mathbbm{R})$ and $P_X(t,\xi)\in L^2(\mathbbm{P})$. Thus, by integrability of $E|P_X(t,\xi)|^2$, which is proved in Appendix A, the mean square error integral $J(\phi)$ is finite. Minimization of $J(\phi)$ gives \cite{maleki}
\begin{equation}\label{phii}
\phi_{opt}(\theta, \tau)=\frac{|A_{E,X}(\theta, \tau)|^2}{E|A_X(\theta, \tau)|^2}I_U(\theta, \tau),
\end{equation}
where $I_U(\theta, \tau)$ denotes the indicator function for the open set
$$U=\{(\theta, \tau);E|A_X(\theta, \tau)|^2>0\}\subset \mathbbm{R} \times \mathbbm{R}^+.$$
The optimal time-frequency kernel is computed by
$$\Phi_{opt}(t, \xi) = \mathcal{M}_1^{-1}\mathcal{M}_2\phi_{opt}(t,\xi).$$
There is no guarantee that the optimal kernel $\phi_{opt}$ is a member of $S_0(\mathbbm{R} \times \mathbbm{R}^+)$. Since we need $\phi \in S_0(\mathbbm{R} \times \mathbbm{R}^+)$ to ensure that our formulas are true, we have to follow a similar approximation procedure demonstrated in \cite{wahl2} to be reconciled by the scale invariant property of the process.

\vspace{0.1 in}
\input{epsf}
\epsfxsize=3.5in \epsfysize=1.5in
\begin{figure}\label{Vector}
\vspace{-0.1 in}
\centerline{\includegraphics[width=1\textwidth]{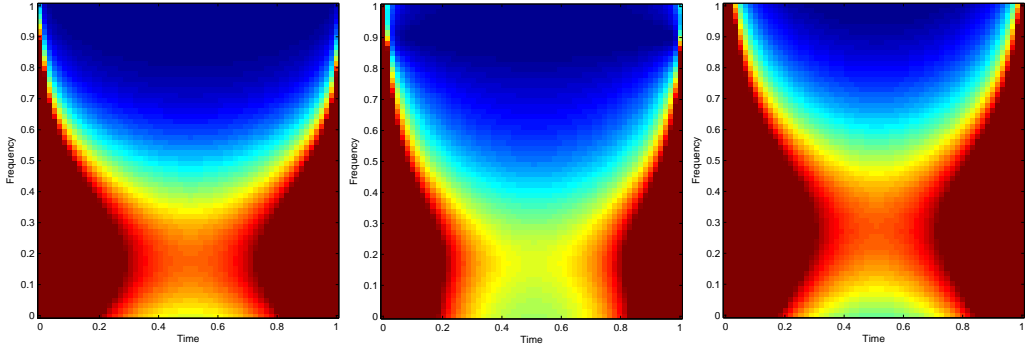}}
\vspace{-0.1in}
\caption{\scriptsize Left: The scale invariant Wigner spectrum of a Gaussian locally self-similar process for given $Q$ and $C$ in (\ref{qc}), with $H=0.5$, $c=4$. Middle: The scale invariant Wigner spectrum estimation of the Gaussian locally self-similar process, by the scale invariant multiple windows method. Right: The scale invariant Wigner spectrum estimation of the Gaussian locally self-similar process, by the quasi Lamperti transform of Hermite functions method. }
\end{figure}
The optimal ambiguity domain kernel, $\phi_{opt}(\theta, \tau)$, for the circularly symmetric Gaussian locally self-similar, locally self-similar chirp, and multicomponent locally self-similar processes with covariance functions (\ref{1''}), (\ref{lsscp}) and (\ref{mlssp}) respectively, are plotted in Figure 1, for different values of $H$ and $c$, where
\begin{equation}\label{qc}
Q(\tau) = \tau^{2H - 1/2 \ln \tau}, \qquad C(\tau) = \tau^{-\frac{c}{8} \ln \tau}, \; c \geq 1.
\end{equation}
The last kernel, $\phi_{opt}^{MLSSCP}(\theta,\tau)$ (Bottom right), is the optimal kernel of a multicomponent locally self-similar chirp process (MLSSCP), which is derived as a combination of the optimal kernels for locally self-similar chirp process and multicomponent locally self-similar process \cite{maleki}. Furthermore, the scale invariant Wigner spectrum of a Gaussian locally self-similar process with covariance function (\ref{1''}), for given $Q$ and $C$, is computed and depicted in Figure 2 (Left) for $H=0.5$, $c=4$.

\section{Multiple Window SIWS Estimation}\label{multitaper}
In this Section, the scale invariant Wigner spectrum is estimated using the Thomson multitaper method \cite{Thomson}. By this method, instead of calculating the ambiguity domain kernel, the SIWS estimation can be computed more efficiently. First, we define a scale invariant windowed spectrogram of a self-similar process $X$, with respect to a window function $\psi$  as

\begin{equation}
\mathcal{G}_{\psi}X(t, \xi) = \bigg{|}\int_0^{\infty} X(s) \psi^*(t/s) s^{-i2\pi \xi-1} ds \bigg{|}^2,
\end{equation}
which can be interpreted as the modulus square of a short-time Mellin transform. The short-time Mellin transform $V_{\psi}(t,\xi)=\int_{0}^{\infty} X(s) \psi^*(t/s) s^{-i2\pi \xi-1}$
$ds$
is a finite variance circularly symmetric Gaussian process which follows from $R_X \in S_0(\mathbbm{R}^+ \times \mathbbm{R}^+)$, $\psi \in S_0(\mathbbm{R}^+)$, the invariance properties of $S_0$ \cite{wahl3} and the assumed circular symmetry of $X$.

Next, some conditions are provided which implies that the Cohen's class counterpart can be replaced by a weighted sum of scale invariant windowed spectrograms. Let
$$\Psi(t_1, t_2):=\kappa^{-1}\circ \mathcal{M}_2^{-1}\Phi(t_1,t_2)$$
$$\;\;\;\;\;\;\;\;\;\;\;\;\;\;\;=\kappa^{-1}\circ \mathcal{M}_1^{-1}\phi(t_1,t_2)$$
\begin{equation}\label{32}
\qquad\qquad\quad\quad\;\;\;=\int_{-\infty}^\infty \phi(\theta,\frac{t_1}{t_2})\sqrt{t_1t_2}^{\; i2\pi\theta} d\theta,
\end{equation}
where $\Phi=\mathcal{M}_1^{-1}\mathcal{M}_2\;\phi$. By the Feichtinger algebra properties \cite{wahl3}, and the correspondence between Fourier and Mellin transforms, since $\Phi, \phi \in S_0$, then their partial Mellin transforms give $S_0$ functions and the above integrals are well defined. Thus, we have the equivalences
$$\phi(\theta,\tau)=\phi^*(-\theta,\frac{1}{\tau}) \Longleftrightarrow \Psi(t_1, t_2)=\Psi^*(t_2, t_1)$$
\begin{equation}\label{33}
\;\;\;\;\;\qquad\qquad \qquad\Longleftrightarrow \Phi(t,\xi)=\Phi^*(t,\xi).
\end{equation}
\noindent For validity of (\ref{33}), see Appendix B.

Now, if $\Phi \in S_0(\mathbbm{R}^{+}\times \mathbbm{R})$ and $\Phi$ is real-valued, then from the invariance properties of $S_0$, and by the fact that the Mellin transform is the abstract Fourier transform of \cite{segal} when we work on the LCA group $\mathbbm{R}^{+}\times \mathbbm{R}^{+}$, Then, $\Psi \in S_0(\mathbbm{R}^{+}\times \mathbbm{R}^{+})$  and by (\ref{33}) the kernel $\Psi$ satisfies the Hermitian property. Thus, since $S_0 \subset L^2$ \cite{banach}, $\Psi$ is the kernel of compact self-adjoint operator on $L^2(\mathbbm{R}^{+}\times \mathbbm{R}^{+})$ \cite{simon}. Therefore, there exist a set of square summable eigenvalues $\lambda_k \in \mathbbm{R}^+$ and eigenfunctions $\psi_k \in L^2(\mathbbm{R}^+)$, such that \cite{wahlberg2}
\begin{equation}
\Psi(t_1, t_2)=\sum_{k=1}^\infty \lambda_k \psi_k(t_1) \psi^*_k(t_2).\label{34}
\end{equation}

Next Theorem shows that, under these circumstances, the Cohen's class counterpart can be rewritten as a weighted sum of scale invariant windowed spectrograms.

\begin{theorem}\label{spectrogram}
Suppose that $X$ is a complex-valued circularly symmetric Gaussian scale invariant process with covariance function $R_X \in S_0(\mathbbm{R}^{+}\times \mathbbm{R}^{+})$. Let $\Phi$ be real-valued and $\Phi \in S_0(\mathbbm{R}^{+}\times \mathbbm{R})$, let  $\Psi$ be defined by (\ref{32}) and fulfill (\ref{34}). Then, we have the equality
$$P_X(t,\xi) = \mathcal{G}(t, \xi), \qquad \forall (t,\xi) \in \mathbbm{R}^+ \times \mathbbm{R}$$
where $P_X$ is a member of Cohen's class counterpart of time-frequency representations and
$$\mathcal{G}(t, \xi) =  \sum_{k=1}^\infty \lambda_k  \mathcal{G}_{\psi_k}X(t, \xi)$$
\begin{equation}\label{spec}
\qquad\qquad\qquad\qquad\qquad\qquad\quad= \sum_{k=1}^\infty \lambda_k \bigg {|}\int_0^\infty X(s)\psi^*_k(\frac{t}{s})s^{-i2\pi\xi-1} ds\bigg {|}^2,\;\;
\end{equation}
that only a finite number of $\lambda_k$ are nonzero.
\end{theorem}

\begin{proof}
See Appendix C.
\end{proof}

Thus, the scale invariant Wigner spectrum can be estimated as a  weighted sum of scale invariant windowed spectrograms. Furthermore, it is shown that, the multitaper spectrogram method, is an effective solution from accomplishment aspects.

\subsection{Example 1}
Let $\{X(t), t \in \mathbbm{R}^+\}$ be a Gaussian locally self-similar process with covariance function (\ref{1''}), where $Q$ and $C$ are given in (\ref{qc}). The scale invariant Wigner spectrum can be estimated using the multitaper method, Eq (\ref{spec}), in which the windows $\psi_k$ and weights $\lambda_k$ are obtained from the equation
$$\mathcal{A}_{opt} \; \psi_k = \lambda_k \; \psi_k, \qquad k = 1, \cdots, M $$
where $\mathcal{A}_{opt}$ is the $M \times M$ sampled matrix corresponding to the kernel  $\Psi(t_1, t_2)$ in Eq (\ref{32}), and the optimal kernel $\phi_{opt}^{LSSP}(\theta, \tau)$ is computed by  (\ref{phii}).  Also, the eigenvalues are ordered according to $\lambda_1 \geq \lambda_2 \geq \cdots \geq \lambda_M$. By the multitaper method, the scale invariant Wigner spectrum is estimated and depicted in Figure 2 (Middle), for $H=0.5$ and $c=4$.

\section{SIWS Estimation Using Hermite Functions}\label{approximation}
In this Section, the scale invariant Wigner spectrum is estimated for Gaussian locally self-similar processes using a set of Hermite functions. By this method, the windows $\psi_k$ in (\ref{spec}), are approximated by the quasi Lamperti transform \cite{spec-rez} of Hermite functions. Therefore, instead of calculating the windows $\psi_k$ by Eq (\ref{34}), the quasi Lamperti transform  of the Hermite functions $h_n(t)$, are considered as:
$$\mathcal{H}_{n}(t) := \mathcal{L}_{H,\alpha} h_n(t) = t^H h_n(\log_{\alpha} t),  \qquad n = 0, 1, \cdots, $$
where
$$h_n(t) = (-1)^n (2^n n ! \sqrt{\pi})^{-1/2} e^{t^2/2} \frac{\partial^n}{\partial t^n}(e^{-t^2}), \qquad n = 0, 1, \cdots.$$
is the set of Hermite functions. Such an approximation is beneficial in calculation, since instead of solving a possibly large-scale eigenvalue problem, only a limited number of Hermite functions should be calculated \cite{sandsten}.

\subsection{Example 2}
Consider a circularly symmetric Gaussian locally self-similar process corresponding to the covariance function (\ref{1''}), with $Q$ and $C$ in (\ref{qc}). Using a set of quasi Lamperti transform of Hermite functions, the scale invariant Wigner spectrum is estimated for $H=0.5, c=4$. The result is shown in Figure 2 (Right). Also, the three first quasi Lamperti transform of Hermite functions are depicted in Figure 3, for $H=0.1, 0.5$ and $0.9$. Since, the optimal multitapers of the proposed class of LSSPs are approximated by a set of quasi Lamperti transform of Hermite functions, to show the accuracy of estimation, the square error
$$e_k(c, H) = (h_k - \psi_k)^T(h_k - \psi_k),$$
is computed for all sets of  Hermite functions with $H$ varying from 0.1 to 0.9. The logarithm of errors for the 1st, 2nd, and 3th eigenvectors are presented in Figure 4 for different $c$ values.

\vspace{0.1in}
\input{epsf}
\epsfxsize=3.5in \epsfysize=1.5in
\begin{figure*}\label{threeHermite}
\vspace{-0.15 in}
\includegraphics[width=1\textwidth]{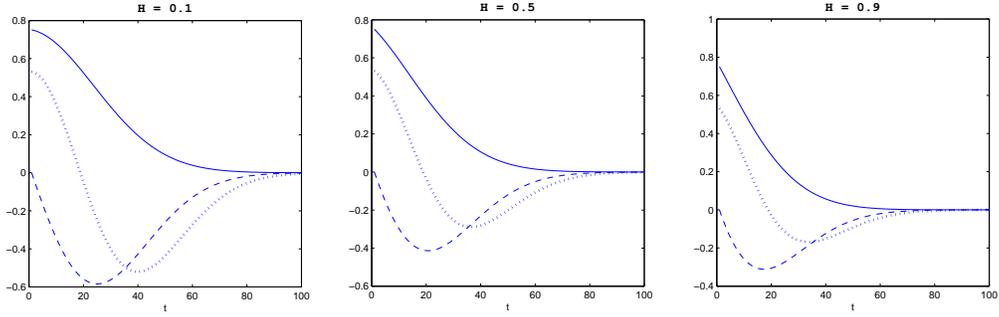}
\vspace{-0.4in}
\caption{\scriptsize  The three first Lamperti transform of Hermite functions for $H = 0.1$ (Left), $H = 0.5$ (Middle), $H = 0.9$ (Right). $\mathcal{H}_1(t)$: solid line, $\mathcal{H}_2(t)$:  dash-dotted line, $\mathcal{H}_3(t)$: dotted line.}
\end{figure*}
\vspace{0.1in}
\input{epsf}
\epsfxsize=3in \epsfysize=1in
\begin{figure*}\label{MSEHermite}
\vspace{-0.1 in}
\includegraphics[width=1\textwidth]{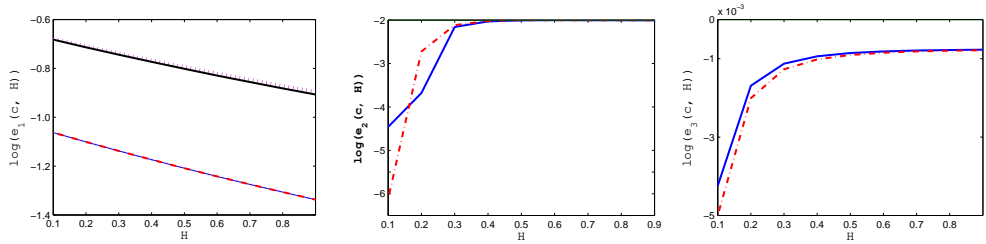}
\vspace{-0.4in}
\caption{\scriptsize  The logarithm of errors between the optimal eigenvectors of a Gaussian LSSP for given $Q$ and $C$ in (\ref{qc}) and the corresponding quasi Lamperti transform of Hermite functions for different $c$ values. The 1st, 2nd, and 3th eigenvector are compared: $c = 4$ (blue solid line), $c = 7$ (red dash-dotted line), $c = 10$ (black solid line) and $c = 20$ (violet dotted line). The logarithm of errors for $c = 10, 20$ in the 2nd and 3th eigenvector are zero. }
\end{figure*}

\section{Evaluation}\label{evaluation}
Now, the predominance of the proposed methods are investigated over using the scale invariant Wigner distribution in SIWS estimation of Gaussian locally self-similar, locally self-similar chirp and multicomponent locally self-similar processes in Tables \ref{lssp1}-\ref{mlssp1}, respectively. In the first simulation, the MSEs of the scale invariant multiple windows method, the quasi Lamperti transform of Hermite functions method, and the SIWD are calculated and presented in Table \ref{lssp1} for three different Gaussian locally self-similar processes corresponding to the covariance function (\ref{1''}) with $Q$ and $C$ in (\ref{qc}) and $H=0.5$, $c=7, 10, 20$, which show that the proposed methods gives the minimum MSE.

The next simulation evaluates the performance of the estimation methods over the SIWD, for Gaussian locally self-similar chirp processes, with covariance function (\ref{lsscp}), where $a=2, b=-2$, and the same $c$ values as in Table \ref{lssp1}. The MSEs are depicted in Table \ref{lsscp1}. Finally, the last simulation with results shown in Table \ref{mlssp1}, evaluates the performance of the estimation methods for Gaussian multicomponent locally self-similar process with covariance function (\ref{mlssp}), the first case with $c_1=4, c_2=7$, the second case with $c_1=4, c_2=10$, and the third case with $c_1=4, c_2=20$. The results show the superiority of the scale invariant multiple windows and the quasi Lamperti transform of Hermite functions methods in spectrum estimation of scale invariant processes.

\begin{table}
\caption{\small{The mean square error of the scale invariant multiple window (SIMW) method and the quasi Lamperti transform of Hermite functions method in SIWS estimation of Gaussian locally self-similar processes, compared to the scale invariant Wigner distribution (SIWD) method, for $H = 0.5$ and $c = 7, 10, 20$.
}}\label{lssp1}
\begin{tabular}{llll}
\hline\noalign{\smallskip}
Method  \qquad \qquad & $c=7$ \qquad \qquad & $c=10$ \qquad \qquad &  $c=20$\\
\noalign{\smallskip}\hline\noalign{\smallskip}
SIMW \qquad \qquad &  $0.24$ \qquad \qquad &  $0.14$ \qquad \qquad &   $0.16$ \\
Hermite \qquad \qquad & $0.93$ \qquad \qquad & $0.87$ \qquad \qquad & $0.83$\\
SIWD \qquad \qquad & $16.79$ \qquad \qquad & $18.67$ \qquad \qquad & $44.92$\\
\noalign{\smallskip}\hline
\end{tabular}
\end{table}
\begin{table}
\caption{\small{The mean square error of the scale invariant multiple windows (SIMW) method and the quasi Lamperti transform of Hermite functions method in SIWS estimation of Gaussian locally self-similar chirp processes with $a = 2, b = -2$, compared to the scale invariant Wigner distribution (SIWD) method, for $H = 0.5$ and $c = 7, 10, 20$.}}\label{lsscp1}
\begin{tabular}{llll}
\hline\noalign{\smallskip}
Method \qquad \qquad  & $c=7$  \qquad \qquad  & $c=10$ \qquad \qquad  & $c=20$\\
\noalign{\smallskip}\hline\noalign{\smallskip}
SIMW \qquad \qquad  &   $0.54$ \qquad \qquad  & $0.90$ \qquad \qquad  & $2.53$ \\
Hermite \qquad \qquad  & $1.03$ \qquad \qquad  & $1.36$ \qquad \qquad  & $2.97$\\
SIWD \qquad \qquad  & $16.81$ \qquad \qquad  & $18.90$ \qquad \qquad  & $45.2$\\
\noalign{\smallskip}\hline
\end{tabular}
\end{table}

\begin{table}
\caption{\small{The mean square error of the scale invariant multiple windows (SIMW) method and the quasi Lamperti transform of Hermite functions method in SIWS estimation of Gaussian multicomponent locally self-similar process, compared to the scale invariant Wigner distribution (SIWD) method, in three cases $c_1 = 4, c_2 = 7$; $c_1 = 4, c_2 = 10$; $c_1 = 4, c_2 = 20$  for $H_1 = 0.2,  H_2 = 0.8$.}}\label{mlssp1}
\label{tab:1}
\begin{tabular}{llll}
\hline\noalign{\smallskip}
Method \quad \quad  & $c_1$ = 4, $c_2$ = 7 \quad \quad  & $c_1$ = 4, $c_2$ = 10 \quad \quad   & $c_1$ = 4, $c_2$ = 20\\
\noalign{\smallskip}\hline\noalign{\smallskip}
SIMW \quad \quad  & $0.16$ \quad \quad  & $0.15$ \quad \quad  & $0.27$\\
Hermite \quad \quad  & $0.57$ \quad \quad & $0.53$ \quad \quad & $0.62 $\\
SIWD \quad \quad  & $26.19$ \quad \quad & $29.88$ \quad \quad & $59.06 $\\
\noalign{\smallskip}\hline
\end{tabular}
\end{table}
The simulation results show that, although locally self-similar processes are a subclass of non-stationary processes, but the scale invariant property exists in such processes, makes them different from the other non-stationary ones; and ordinary non-stationary spectrum estimation methods may not be applicable for time-frequency analysis of such processes. So, some special tools should be considered, which are compatible with the scale invariant property.

\section{Conclusion}
The mean square error optimal scale invariant multitaper spectrogram estimator, for a class of Gaussian locally self-similar processes is evaluated and compared to the scale invariant Wigner distribution, which is a classical estimator of SIWS. An evaluation is also made for the classes of locally self-similar chirp processes, and multicomponent locally self-similar processes. The results show that, dealing with scale invariant processes, current methods of non-stationary spectrum estimation should be modified to be reconciled to the scale invariant property. Moreover, it is shown that, the windows in multitaper method are well approximated by the quasi Lamperti transform of a set of Hermite functions, which is beneficial in calculation.

\section*{Appendices}
\noindent{\textbf{Appendix A}}:
To prove the integrability of $E|P_X(t, \xi)|^2$, we should show that the following integral is finite
$$\int_{-\infty}^{\infty} \int_{0}^{\infty} E|\int_{-\infty}^{\infty} \int_{0}^{\infty} A_X(\theta, \tau) \phi(\theta, \tau) t^{i 2\pi \theta} \tau^{-i 2\pi \xi-1} d\tau d\theta |^2 dt d\xi \qquad\qquad\qquad$$
$$\;\; = \int_{-\infty}^{\infty} \int_{0}^{\infty} E \int_{-\infty}^{\infty} \int_{-\infty}^{\infty} \int_{0}^{\infty} \int_{0}^{\infty} A_X(\theta_1, \tau_1) A^*_X(\theta_2, \tau_2) \phi(\theta_1, \tau_1) \phi^*(\theta_2, \tau_2) \qquad $$
$$ \times \; t^{i 2\pi (\theta_1-\theta_2)} (\tau_1/\tau_2)^{-i 2\pi \xi} \frac{d\tau_1}{\tau_1} \frac{d\tau_2}{\tau_2} d\theta_1 d\theta_2  dt d\xi. \qquad \qquad \qquad \qquad\qquad \qquad \qquad $$
Now, by $\int_{0}^{\infty} t^{-i 2 \pi u-1} dt = \delta(u)$ and $\int_{-\infty}^{\infty} v^{-i 2 \pi \xi} d\xi = \delta(\ln v)$, and by the fact that $A_X(\theta, \tau)$ is a second order stochastic process, the integrability of $E|P_X(t, \xi)|^2$ is concluded.\\

\noindent{\textbf{Appendix B}}: For validity of (\ref{33}), let
 $\phi(\theta,\tau)=\phi^*(-\theta,\frac{1}{\tau})$, so

$$\Psi^*(t_2, t_1)=\int_{-\infty}^\infty \phi^*(\theta\;,\frac{t_2}{t_1})\sqrt{t_1t_2}^{\;-i2\pi\theta}=\int_{-\infty}^\infty \phi(\theta\;,\frac{t_1}{t_2})\sqrt{t_1t_2}^{\;i2\pi\theta}=\Psi(t_1, t_2),$$

$$\Phi^*(t,\xi)=\int_{-\infty}^\infty \int_0^\infty \phi^*(\theta, \tau) \; t^{-i2\pi\theta} \tau^{i2\pi\xi-1} d\tau d\theta$$
$$\qquad\qquad\;\;=\int_{-\infty}^\infty \int_0^\infty \phi^*(-\theta,\frac{1}{\tau}) \; t^{i2\pi\theta} \tau^{-i2\pi\xi-1} d\tau d\theta$$
$$\qquad\qquad =\int_{-\infty}^\infty \int_0^\infty \phi(\theta, \tau) \; t^{i2\pi\theta} \tau^{-i2\pi\xi-1} d\tau d\theta \;\;\;$$
$$  =\Phi(t,\xi).\qquad\qquad\qquad\qquad\;\;\;\;$$

\noindent {\textbf{Appendix C}}: Proof of Theorem \ref{spectrogram}. The Cohen's class counterpart is represented as
$$P_X(t,\xi)=\int_{-\infty}^\infty\int_0^\infty A_X(\theta, \tau)\phi(\theta,\tau) \; t^{i2\pi\theta} \tau^{-i2\pi\xi-1} d\tau d\theta \qquad\qquad\qquad\qquad$$
$$\qquad\qquad\quad=\int_{-\infty}^\infty \int_0^\infty \int_0^\infty X(u\sqrt{\tau})X^*(u/\sqrt{\tau}) \phi(\theta,\tau) (\frac{t}{u})^{i2\pi\theta} \tau^{-i2\pi\xi-1} \frac{du}{u} d\tau d\theta.\;$$
Let $t_1:=u\sqrt{\tau}, t_2:=u/\sqrt{\tau}$, then using (\ref{32}), we have that
$$P_X(t,\xi) = \int_0^\infty\int_0^\infty X(t_1)X^*(t_2) \Psi(\frac{t}{t_2}\;, \frac{t}{t_1})t_1^{-i2\pi\xi-1}\;t_2^{i2\pi\xi-1}dt_1 dt_2,$$
where for a Hermitian kernel $\Psi$, Eq (\ref{34}) completes the proof.

\end{document}